\documentclass[a4paper,10pt]{article}
\usepackage{latexsym,amssymb,amsfonts,amsmath,amsthm,mathrsfs,amstext,color,graphicx,times}

\usepackage[mathscr]{euscript}
\usepackage{indentfirst}

\newcommand{\R}{\mathbb{R}}
\newcommand{\N}{\mathbb{N}}
\newcommand{\Q}{\mathbb{Q}}

\usepackage{wrapfig}
\usepackage[all]{xy}
\usepackage{tikz}
\usetikzlibrary{patterns}

\usepackage{enumitem}

\newtheorem{theorem}{Theorem}[section]

\newtheorem{proposition}{Proposition}[section]
\newtheorem{remark}{Remark}
\newtheorem{lemma}{Lemma}[section]

\newcommand{\tos}{\rightrightarrows} 

\title{A note on quasi-equilibrium problems}
\author{
John Cotrina
	\thanks{Universidad del Pac\'ifico. Av. Salaverry 2020, Jes\'us Mar\'ia, Lima, Per\'u. Email: 
	\texttt{\{ cotrina\_je,~zuniga\_jj\}@up.edu.pe}} 
	\and Javier Z\'u\~niga\footnotemark[1]
}

\begin{document}
\maketitle
\begin{abstract}
The purpose of this paper is to prove the existence of solutions of quasi-equilibrium problems
without any generalized monotonicity assumption. Additionally, we give an application to quasi-optimization problems.
\bigskip


\bigskip


\end{abstract}

\section{Introduction and definitions}

Given a real topological vector space $X$, a subset $C$ of $X$, a bifunction $f:C\times C\to\R$ and
a set-valued map $K:C\tos C$, the \emph{quasi-equilibrium problem} (QEP) consists in finding
\begin{align}\label{QEP}
x\in K(x)\mbox{ such that }f(x,y)\geq0,\mbox{ for all }y\in K(x). 
\end{align}
When $K(x)=C$ for any $x\in C$, the QEP coincides with the classical \emph{equilibrium problem}, which was introduced
by Oettli and Blum in \cite{OB93}, and has been extensively studied in recent years (see for instance 
\cite{BP01,castellani2012,SN10} and the references therein). 

The classical example of quasi-equilibrium problem is the quasi-variational inequality problem, which consists in finding
$x\in K(x)$, such that there exists $x^*\in T(x)$ with $\langle x^*,y-x\rangle\geq0,\mbox{ for all }y\in K(x)$ where $T:X\tos X^*$ is a set-valued map, $X^*$ is the dual space of $X$ and $\langle\cdot,\cdot\rangle$ denotes
the duality paring between $X$ and $X^*$. So, if $T$ has compact values, and we define 
the {\em representative bifunction} $f_T$ of $T$ by
\[
f_T(x,y)=\sup_{x^*\in T(x)}\langle x^*,y-x\rangle, 
\]
it follows that every solution of the QEP associated to $f_T$ and $K$ is a solution 
of the quasi-variational inequality problem associated to $T$ and $K$, and conversely. 

Recently in \cite{ACI} the authors show existence of solution of the QEP using generalized monotonicity 
for $f$ in a finite dimensional space. Castellani and Giuli (\cite{castellani_Giuli15}) proved an existence result 
which does not involve any monotonicity assumption of $f$ in separable Banach  spaces. 

The aim of this note is to show existence of solution for the QEP without generalized monotonicity assumptions
but for Hausdorff locally convex real
topological vector spaces which generalizes the spaces in \cite{ACI,castellani_Giuli15}.

\section{Existence results}\label{S3}
Our existence result will be obtained as a consequence of Kakutani's Fixed Point Theorem which is stated in the next result and it can be found in \cite{GD}.
\begin{theorem}[Kakutani's theorem]\label{Kakutani}
Let $C$ be a nonempty compact convex subset of a locally convex space $X$ and 
let $S: C \tos C$ be a set-valued map. If $S$ is upper semicontinuous such that for all $x\in C$, $S(x)$ is nonempty, closed and convex, 
then $S$ admits a fixed point.
\end{theorem}

We denote by QEP$(f,K)$ the solution set of \eqref{QEP}
and we define the following set-valued map $S:C\tos C$ by
\[
S(x)=\{x_0\in K(x):~f(x_0,y)\geq0\mbox{ for all }y\in K(x)\}.
\]

The proposition below plays an important role in our existence result.
\begin{proposition}\label{closed-S}
Let $f:C\times C\to\R$ be a bifunction and let $K:C\tos C$ be a set-valued map,
where $C$ is a nonempty subset of a real topological vector space $X$.
If $K$ is closed and lower semicontinuous and $\{(x,y)\in C\times C:~f(x,y)\geq0\}$ is closed; then $S$ is closed.
\end{proposition}
\begin{proof}
Let $(x_n,z_n)_{n\in\N}$ be a sequence contained in the graph of $S$ converging to $(x_0,z_0)$. 
Since $K$ is closed, we have $z_0\in K(x_0)$. The lower semicontinuity of $K$ implies that for any $y\in K(x_0)$, there exists $(y_n)_{\in\N}$ converging to $y$ such that $y_n\in K(x_n)$, for all $n\in\N$. Additionally, as $z_n\in S(x_n)$ we have $f(z_n,y_n)\geq0$ for every $n\in\N$, which in turn implies by hypothesis that $f(z_0,y)\geq0$. Therefore, $z_0\in S(x_0)$.
\end{proof}

We finish this section with our main existence result.

\begin{theorem}\label{main}
Let $f:C\times C\to\R$ be a bifunction and let $K:C\tos C$ be a set-valued map,
where $C$ is a compact convex and nonempty subset of a Hausdorff locally convex real topological vector space $X$.
If the following hold:
\begin{enumerate}
 \item[$i)$] $K$ is closed and lower semicontinuous with convex values;
 \item[$ii)$] $\{x\in C:~f(x,y)\geq0\}$ is convex, for every $y\in C$;
\item[$iii)$] for any subset $\{x_1,\dots,~x_n\}$ of $C$, and any $x \in {\rm co}(\{x_1,\dots,~x_n\})$ 
(here {\rm co} is the convex hull), $\max_{i=1,\dots,n}f(x,x_i)\geq0$;
 \item[$iv)$] $\{(x,y)\in C\times C:~f(x,y)\geq0\}$ is closed;
\end{enumerate}
then {\rm QEP}$(f,K)$ is nonempty.
\end{theorem}

For the prove of the previous theorem we need the following result.

\begin{theorem}\cite[Theorem 2.3]{SN10}\label{SN}
Let $f:C\times C\to\R$ be a bifunction, where $C$ is a compact convex and nonempty subset of a Hausdorff real topological vector space $X$.
If for any $\{x_1,\dots,~x_n\} \subset C$ and $x \in {\rm co}(\{x_1,\dots,~x_n\})$, $\max_{i=1,\dots,n}f(x,x_i)\geq0$; and  $\{y\in C:~f(x,y)\geq0\}$ is closed, for every $x\in C$; then there exists $x_0\in C$ such that $f(x_0,y)\geq0$, for all $y\in C$.
\end{theorem}

\begin{proof}[Proof of Theorem \ref{main}]
By Proposition \ref{closed-S}, the set-valued map $S$ is closed. For each $x\in C$, $S(x)$ is closed, convex and nonempty due to conditions $ii),~iii)$ and $iv)$, and Theorem \ref{SN}. As $C$ is compact, we have that
 $S$ is upper semicontinuous. Thus, by Kakutani's theorem  $S$ has a fixed point.
\end{proof}

\begin{remark} \label{diagonal}
Notice that quasiconcavity with respect to the first variable of $f$ implies part $ii)$ of Theorem \ref{main}.
Moreover, if $f(x,\cdot)$ is quasiconvex and $f(x,x)=0$ for any $x\in C$ then part $iii)$ of Theorem \ref{main} holds. 
However, the converse is not true in general. Consider for instance $f:[0,1]\times[0,1]\to\R$ defined by
\[
f(x,y)=\left\lbrace\begin{array}{cl}
                    1,&y\in \Q\cap[0,1]\\
                    0,&y\notin \Q
                   \end{array}\right..
 \]
Clearly $f$ satisfies condition $ii)$ of Theorem \ref{main}, but it is not quasiconvex with respect to its second argument
nor it vanishes on the diagonal of $[0,1]\times[0,1]$.
\end{remark}

\section{Application to quasi-optimization}

Given a real-valued function $h:C\to \R$ and a set-valued map $K:C\tos C$, where $C$ is a subset of a Hausdorff locally convex real topological vector space $X$, the {\em quasi-optimization problem} (QOpt) is described as
\[
\mbox{find }x_0\in K(x_0)\mbox{ such that }~~~ \min_{z\in K(x_0)} h(z)=h(x_0).
\]

\par
The terminology of quasi-optimization problems comes from \cite{GMP00} (see 
formula (8.3) and Proposition 12) and has been recently used in \cite{AC-2013,FKa10}. It 
emphasizes the fact that it is not a standard optimization problem since the 
constraint set depends on the solution and it also highlights the parallelism 
to quasi-equilibrium problems.

\begin{remark}
Under continuity of the constraint set-valued map the continuity of the objective function is not a sufficient 
condition to guarantee the existence of solution for the QOpt.
Consider for instance the function $h:[0,2]\to\R$ and the set-valued map $K:[0,2]\tos [0,2]$ both defined by 
\[
h(x)= \left\lbrace\begin{array}{cl}
                   |x-\frac{1}{2}|, &0\leq x\leq 1\\
                   |x-\frac{3}{2}|,& 1<x\leq 2
                  \end{array}\right.
\qquad K(x)=\left\lbrace\begin{array}{cl}
                   [-\frac{3}{2}x+\frac{3}{2},2], &0\leq x\leq 1\\
                   \left[0,-\frac{3}{2}x+\frac{7}{2}\right],&1<x\leq 2
                  \end{array}\right.
\]

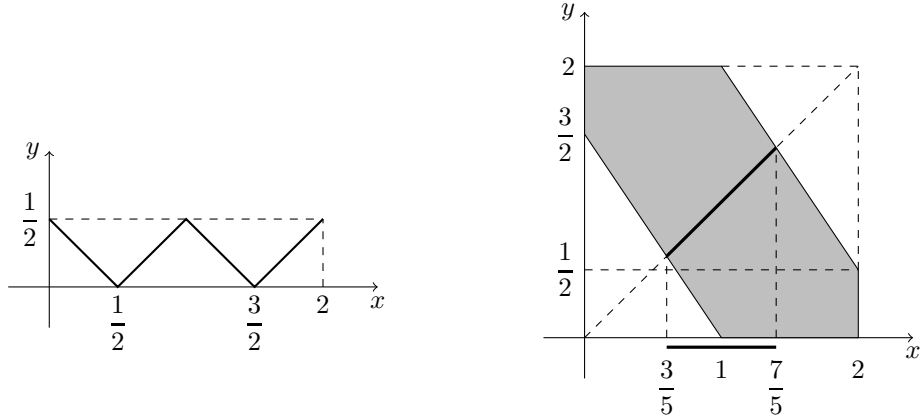
\begin{figure}
\begin{tikzpicture}[scale=1.8]
\fill[color=white](0,0)--(2,0)--(2,-1)--(0,-1)--(0,0);
\draw[->] (-0.3,0)--(2.4,0)node[below]{$x$};
\draw[->] (0,-0.3)--(0,1)node[left]{$y$};
\draw[thick](0,0.5)--(0.5,0) -- (1,0.5) -- (1.5,0)--(2,0.5);
\draw(0.5,0)node[below]{$\displaystyle\frac{1}{2}$};
\draw(1.5,0)node[below]{$\displaystyle\frac{3}{2}$};
\draw(2,0)node[below]{$2$};
\draw(0,0.5)node[left]{$\displaystyle\frac{1}{2}$};
\draw[dashed](2,0)--(2,0.5);
\draw[dashed](0,0.5)--(2,0.5);
\end{tikzpicture}
\hfill
\begin{tikzpicture}[scale=1.8]
\draw[->] (-0.3,0)--(2.4,0)node[below]{$x$};
\draw[->] (0,-0.3)--(0,2.4)node[left]{$y$};
\fill[color=gray!50](0,1.5)--(0,2)--(1,2)--(2,0.5)--(2,0)--(1,0)--(0,1.5);
\draw(0,1.5)--(0,2)--(1,2)--(2,0.5)--(2,0)--(1,0)--(0,1.5);
\draw(1,-0.1)node[below]{$1$};
\draw(2,-0.1)node[below]{$2$};
\draw(0.6,-0.1)node[below]{$\displaystyle\frac{3}{5}$};
\draw(1.4,-0.1)node[below]{$\displaystyle\frac{7}{5}$};
\draw(0,2)node[left]{$2$};
\draw[dashed](0,0.5)--(2,0.5);
\draw(0,0.5)node[left]{$\displaystyle\frac{1}{2}$};
\draw(0,1.5)node[left]{$\displaystyle\frac{3}{2}$};
\draw[dashed](2,0.5)--(2,2);
\draw[dashed](1,2)--(2,2);
\draw[dashed](0,0)--(2,2);
\draw[very thick] (0.6,-0.07)--(1.4,-0.07);
\draw[very thick] (0.6,0.6)--(1.4,1.4);
\draw[dashed](0.6,0)--(0.6,0.6);
\draw[dashed](1.4,0)--(1.4,1.4);
\end{tikzpicture}
\caption{Graphs of $h$ and $K$.} \label{graph}
\end{figure}

Figure \ref{graph} shows the graphs of $h$ and $K$. Clearly, $h$ is continuous and $K$ is closed and lower semicontinuous. The set of fixed points of $K$ is the interval $[3/5,7/5]$. It is not difficult to show that the QOpt does not have a solution.
\end{remark}

Associated to $h$ and $K$, let us define the bifunction $f^h: C\times C\to\R$ by
\[
 f^h(x,y)=h(y)-h(x).
\]
Now, we can characterize the solutions of the QOpt by solutions of the QEP associated to $f^h$ and $K$. We denote by QOpt$(h,K)$ the solution set of the QOpt. The definition of $f^h$ implies the following lemma.
\begin{lemma}\label{QEP-QOpt}
With the previous notation and assuming that $x_0\in C$, then $x_0\in{\rm QEP}(f^h,K)$ if and only if $x_0\in{\rm QOpt}(h,K)$. 
\end{lemma}

Finally, we are ready for our result about the existence of solutions for the QOpt which generalizes \cite[Proposition 4.2 and Proposition 4.5]{AC-2013}.
\begin{theorem}
Let $h:C \to \R$ be a function and 
let $K:C \tos C$ be a set-valued map, where $C$ is a convex and compact subset of a Hausdorff locally convex real topological vector space $X$. 
If $K$ is closed and lower semicontinuous with convex values, and $h$ is quasiconvex and continuous; then {\rm QOpt}(h,K) is nonempty.
\end{theorem}
\begin{proof}
We want to verify all assumptions of Theorem \ref{main}. The first one is trivial. Since $h$ is quasiconvex, the bifunction $f^h$ is quasiconcave with respect to its first argument, which implies $ii)$, and quasiconvex with respect to its second argument. Moreover, $f^h$ vanishes on the diagonal of $C\times C$, hence $f^h$ satisfies condition $iii)$ (see Remark \ref{diagonal}). Additionally, as $h$ is continuous then $f^h$ is continuous and hence condition $iv)$ holds.
Therefore, there exists $x_0\in{\rm QEP}(f^h,K)$. The result follows from Lemma \ref{QEP-QOpt}.
\end{proof}

\end{document}